\documentclass[a4paper, 12pt]{amsart}

\usepackage{amssymb,amsmath}
\usepackage[dvips]{graphics}
\usepackage[dvips]{color}
\usepackage{graphicx}
\usepackage[english]{babel}
\usepackage[latin1]{inputenc}
\usepackage{comment}

\newtheorem{theorem}{Theorem}[section]
\newtheorem{lemma}[theorem]{Lemma}

\newtheorem{corollary}[theorem]{Corollary}

\theoremstyle{definition}

\newtheorem{example}[theorem]{Example}
\newtheorem{remark}[theorem]{Remark}

\numberwithin{equation}{section}

\title[Fractional maximal functions in metric measure spaces]{Fractional maximal functions\\ in metric measure spaces}

\author{Toni Heikkinen}
\author{Juha Lehrb\"ack}
\author{Juho Nuutinen}
\author{Heli Tuominen}

\newcommand\re{\mathbb R}

\newcommand\z{\mathbb Z}

\newcommand\eps{\varepsilon}
\newcommand\M{\operatorname{\mathcal M}}

\providecommand{\ch}[1]{\text{\raise 2pt \hbox{$\chi$}\kern-0.2pt}_{#1}}

\providecommand{\vint}[1]{\mathchoice
          {\mathop{\vrule width 5pt height 3 pt depth -2.5pt
                  \kern -9pt \kern 1pt\intop}\nolimits_{\kern -5pt{#1}}}%
          {\mathop{\vrule width 5pt height 3 pt depth -2.6pt
                  \kern -6pt \intop}\nolimits_{\kern -3pt{#1}}}%
          {\mathop{\vrule width 5pt height 3 pt depth -2.6pt
                  \kern -6pt \intop}\nolimits_{\kern -3pt{#1}}}%
          {\mathop{\vrule width 5pt height 3 pt depth -2.6pt
                  \kern -6pt \intop}\nolimits_{\kern -3pt{#1}}}}
                  
\begin{document}

\begin{abstract}
We study the mapping properties of fractional maximal operators in Sobolev and Campanato spaces in metric measure spaces.  
We show that, under certain restrictions on the underlying metric measure space, 
fractional maximal operators improve the Sobolev regularity of functions and map functions in Campanato spaces to H\"older continuous functions. We also give an example of a space where fractional maximal function of a Lipschitz function fails to be continuous.
\end{abstract}



\maketitle

\section{Introduction}

Fractional maximal operators are standard tools in partial differential equations, potential theory and harmonic analysis. In the Euclidean case they are studied for example in \cite{A1}, \cite{A2}, \cite{AH}, \cite{KS}, \cite{KrKu}, \cite{LMPT}, \cite{MW},  and in the metric setting in \cite{GGKK}, \cite{Go}, \cite{Go2}, \cite{HKNT}, \cite{HT}, \cite{PW}, \cite{SWZ}, \cite{W}.
In the Euclidean case, the fractional maximal operator $\M_\alpha$, defined by
\[
\M_\alpha u(x)=\sup_{r>0}\,\frac{r^{\alpha}}{|B(x,r)|}\int_{B(x,r)}|u(y)|\,dy,
\]
has similar smoothing properties as the Riesz potential, see Kinnunen and Saksman \cite{KS}. 
In this paper, we generalize the Euclidean results of \cite{KS} to the metric setting under an annular decay property, which places certain restrictions on the geometry of the space $X$. More precisely,  Theorem \ref{Lp for M} implies that the fractional maximal function of an $L^p$-function, $p>1$, belongs to a (pointwise fractional) Sobolev space $M^{s,p}(X)$, where $s$ depends on the annular decay. 
Another example of a smoothing property is given in Theorem \ref{thm: sobo}, where we show that the fractional maximal operator maps the Sobolev space $M^{1,p}(X)$ to a slightly better Sobolev space $M^{1,p^*}(X)$, where $p^*$ is the conjugate exponent of      $p$. The proof of this result is based on an unpublished proof of MacManus \cite{M2} for the Sobolev boundedness of the usual Hardy-Littlewood maximal operator.

We also study the action of the fractional maximal function in Campanato spaces. In this context our main result is Theorem \ref{main thm}, which states that the fractional maximal function maps functions in Campanato spaces to H\"older continuous functions, provided that the space satisfies the annular decay property. The result is new even in the Euclidean case. 
This is again analogous to the well known properties of the Riesz potential, studied in the metric setting for example in \cite{GV} and \cite{GSV}. Note here that the Campanato estimates for the Riesz potential do not immediately imply the corresponding oscillation estimates for the fractional maximal function. 
In particular, Theorem \ref{main thm} shows that the fractional maximal operator maps H\"older continuous functions to H\"older continuous functions with a better exponent, and functions of bounded mean oscillation to H\"older continuous functions. 

A part of the motivation for our work comes from \cite{HKNT}, where it was shown that similar mapping properties hold for the so called discrete fractional maximal operator even without the annular decay property. In contrast to those results, we conclude this paper with two examples which verify that there is a real obstruction in the study of fractional maximal functions in metric measure spaces. We modify the example given by Buckley in \cite{B} for the standard Hardy-Littlewood maximal function and show that the fractional maximal function of a Lipschitz continuous function may fail to be continuous if the space does not satisfy annular decay. The same is actually true also for the noncentered fractional maximal function, for which the main results given in Sections \ref{section: campanato} and \ref{section: M} hold as well, as is briefly noted at the end of each theorem.

\section{Notation and preliminaries}\label{preliminaries}
We assume throughout the paper that $X=(X, d,\mu)$ is a metric measure space
equipped with a metric $d$ and a Borel regular, doubling outer
measure $\mu$, for which the measure of every open set is positive and
the measure of each bounded set is finite.
The doubling property means that there exists a fixed constant
$c_d>0$, called the doubling constant, such that 
\begin{equation}\label{doubling measure}
\mu(B(x,2r))\le c_d\mu(B(x,r))
\end{equation}
for every ball $B(x,r)=\{y\in X:d(y,x)<r\}$. 

We say that the measure $\mu$ satisfies a lower bound condition if there exist constants $Q\ge 1$ and $c_{l}>0$ such that
\begin{equation}\label{lower bound}
 \mu(B(x,r))\ge c_{l}r^{Q}
\end{equation}
for all $x\in X$ and $r>0$.

We follow the standard procedure that the letter $C$ denotes a positive constant whose value is not necessarily the same at each occurrence.

\subsection*{The fractional maximal function}
Let $\alpha\ge 0$. 
The fractional maximal function of a locally integrable function $u$ is 
\begin{equation}\label{maximal function}
\M_\alpha u(x)=\sup_{r>0}\,r^{\alpha}\vint{B(x,r)}|u|\,d\mu,
\end{equation}
where $u_B=\vint{B}u\,d\mu=\frac1{\mu(B)}\int_Bu\,d\mu$ is the integral average of  $u$ over $B$.
For $\alpha=0$, we have the usual Hardy-Littlewood maximal function, $\M u=\M_0 u$, 
\[
\M u(x)=\sup_{r>0}\,\vint{B(x,r)}|u|\,d\mu. 
\]

\subsection*{Annular decay properties}
Let $0<\delta\le 1$. 
We say that the metric measure space $X$ satisfies the $\delta$-annular decay property, if there exists
a constant $C>0$ such that for all $x\in X$, $R>0$, and $0<h<R$, we have
\begin{equation}\label{B annulus}
\mu\big(B(x,R)\setminus B(x,R-h)\big)\le C\Big(\frac h R\Big)^\delta\mu(B(x,R)).
\end{equation}

Similarly, we say that $X$ satisfies the relative $\delta$-annular decay property, 
if there exist a constant $C>0$ such that for all $x\in X$, $R>0$, and $0<h<R$, we have
\begin{equation}\label{MM annulus}
\mu\big(B\cap(B(x,R)\setminus B(x,R-h))\big)\le C\Big(\frac h {r_B}\Big)^\delta \mu(B)
\end{equation}
for all balls $B$ with radius $r_B\le 3R$.
Note that the relative condition \eqref{MM annulus} 
implies the $\delta$-annular decay property \eqref{B annulus}.

If $X$ is a geodesic space (or more generally a length space), then $X$ satisfies the relative annular decay for some $\delta>0$.
See for instance \cite{B}, \cite[Chapter 9]{HjK} and \cite{Rou} for this fact, examples and for more information on these and related conditions.

Buckley studied in \cite{B}  the action of the usual maximal operator $\M$ on 
$C^{0,\beta}(X)$,
the space of $\beta$-H\"older continuous functions equipped with the seminorm 
\[
\|u\|_{C^{0,\beta}(X)}=\sup_{x\neq y} \frac{|u(x)-u(y)|}{d(x,y)^{\beta}}.
\]
He showed that if a doubling space $X$ satisfies the $\delta$-annular decay property, then 
\begin{equation}\label{c beta->c beta}
\M\colon C^{0,\beta}(X)\to C^{0,\beta}(X)
\end{equation}
is bounded whenever $0<\beta\le\delta$. 
He also gave an example of a doubling space where the maximal function
of a Lipschitz function is not continuous; see Section \ref{section: example}.

\section{Campanato spaces}\label{section: campanato}
In this section, we study the action of the fractional maximal operator $\M_\alpha$ on 
Campanato spaces $\mathcal L^{p,\beta}(X)$. Let $p\ge 1$ and $\beta\in\mathbb{R}$.
A locally integrable function $u$ belongs to $\mathcal L^{p,\beta}(X)$, if
\[
\|u\|_{\mathcal L^{p,\beta}(X)}
=\sup r^{-\beta}\Big(\,\vint{B(x,r)}|u-u_{B(x,r)}|^p\,d\mu\Big)^{1/p}<\infty,
\]
where the supremum is taken over all $x\in X$ and $r>0$. 

If $\beta<0$, the space $\mathcal L^{p,\beta}(X)$ coincides with the Morrey space $L^{p,\beta}(X)$ consisting of functions $u$ for which
\[
\|u\|_{L^{p,\beta}(X)}
=\sup_{x\in X,r>0} r^{-\beta}\Big(\,\vint{B(x,r)}|u|^p\,d\mu\Big)^{1/p}<\infty,
\]
see for example \cite[Theorems 2.1, 2.2 and Corollary 2.3]{Na}.

If $\beta=0$, then $\mathcal L^{p,0}(X)=\mathcal L^{1,0}(X)=\text{BMO}(X)$, the space of functions of bounded mean oscillation. Moreover,
\begin{equation}\label{Lp0=BMO}
\|u\|_{\text{BMO}(X)}\le\|u\|_{\mathcal L^{p,0}(X)}\le C_p\|u\|_{\text{BMO}(X)},
\end{equation}
where $C_p$ depends on the doubling constant and $p$. Here, the first inequality follows directly from the H\" older inequality and the second inequality is a consequence of the John--Nirenberg theorem, see \cite[Theorem 2.2]{B2} and \cite[Corollary 3.10]{GCRDF}.

If $\beta>0,$ then 
\begin{equation}\label{Lpbeta=Holder}
C^{-1}\|u\|_{C^{0,\beta}(X)}\le\|u\|_{\mathcal L^{p,\beta}(X)}\le \|u\|_{C^{0,\beta}(X)},
\end{equation}
where $C$ depends on the doubling constant and $\beta$,  see for example \cite[Theorem 2.4]{Na}.

\begin{theorem}\label{main thm}
Assume that $X$ satisfies the $\delta$-annular decay property \eqref{B annulus} and that
either $0<\alpha\le \delta$, $\beta\neq 0$ and $0\le \alpha+\beta\le \delta$ or 
$0<\alpha< \delta$ and $\beta=0$.
If $u\in \mathcal L^{p,\beta}(X)$, $p\ge 1$ and $\M_\alpha u\not\equiv\infty$,
then $\M_{\alpha}u\in C^{0,\alpha+\beta}(X)$. 
Moreover, there is a constant $C>0$, independent of $u$, 
such that
\begin{equation}\label{main eq}
\|\M_\alpha u\|_{C^{0,\alpha+\beta}(X)}\le C\|u\|_{\mathcal L^{p,\beta}(X)}.
\end{equation}

\end{theorem}
Note that Theorem \ref{main thm} implies that 
\begin{equation}\label{c beta->c alpha+beta}
\M_\alpha\colon C^{0,\beta}(X)\to C^{0,\alpha+\beta}(X), 
\end{equation}
where $\beta>0,\ \alpha+\beta\le \delta$, and
\begin{equation}\label{bmo->c alpha}
\M_\alpha\colon \text{BMO}(X)\to C^{0,\alpha}(X),
\end{equation}
where $0<\alpha<\delta$,
are bounded operators, when restricted to functions satisfying $\M_\alpha u\not\equiv\infty$. 
If the measure $\mu$ satisfies the lower bound condition \eqref{lower bound}
and $\alpha-Q/p\ge 0$, then also
\begin{equation}\label{Lp->c alpha-Q/p}
\M_\alpha\colon L^p(X)\to C^{0,\alpha-Q/p}(X).
\end{equation}
Notice that \eqref{c beta->c alpha+beta}, \eqref{bmo->c alpha} and \eqref{Lp->c alpha-Q/p}, 
where $\alpha-Q/p>0$, are analogous to the well known properties of the Riesz potential, studied in the metric setting for example in \cite{GV} and \cite{GSV}.

For the proof Theorem \ref{main thm}, we need a lemma which is proved using a chaining argument.
\begin{lemma}[\cite{HKNT}, Lemma 7.1.]
Let $x\in X$, $0<r\le R$ and $y\in B(x,C_0R)$,
and let $u\in \mathcal L^{p,\beta}(X)$.
If $\beta<0$, then
\begin{equation}\label{beta<0}
|u_{B(y,r)}-u_{B(x,R)}|
\le Cr^{\beta}\|u\|_{\mathcal L^{p,\beta}(X)}.
\end{equation}
If $\beta=0$, then
\begin{equation}\label{beta=0}
|u_{B(y,r)}-u_{B(x,R)}|
\le C\log \frac{CR}{r}\|u\|_{\mathcal L^{p,0}(X)}.
\end{equation}
The constant $C$ depends only on the doubling constant, $C_0$ and $\beta$.
\end{lemma}

\begin{proof}[Proof of Theorem \ref{main thm}]
Since $|u|\in \mathcal L^{p,\beta}(X)$ and 
$\||u|\|_{\mathcal L^{p,\beta}(X)}\le C\|u\|_{\mathcal L^{p,\beta}(X)}$, we may assume that $u\ge 0$.
Let $r>0$ and define $v\colon X\to[0,\infty)$ as $v(x)=r^\alpha u_{B(x,r)}$.
We begin by proving the claim for $v$. Let $x,y\in X$.

\noindent {\bf Case 1.} Assume first that $r\le 2d(x,y)$. 
Let $B=B(x,2d(x,y))$, $B_x=B(x,r)$ and $B_y=B(y,r)$. 
If $\beta<0$, then, by \eqref{beta<0},
\[
\begin{split}
|v(x)-v(y)|
&=|r^\alpha u_{B_x}-r^\alpha u_{B_y}|
\le r^\alpha \big(|u_{B_x}-u_B|+|u_B-u_{B_y}|\big)\\
&\le Cr^{\alpha+\beta}\|u\|_{\mathcal L^{p,\beta}(X)}\\
&\le Cd(x,y)^{\alpha+\beta}\|u\|_{\mathcal L^{p,\beta}(X)}.
\end{split}
\]
Similarly, if $\beta=0$,  estimate \eqref{beta=0} implies that
\[
\begin{split}
\big|r^\alpha u_{B_x}-r^\alpha u_{B_y}\big|
&\le Cr^{\alpha}\log\frac{Cd(x,y)}{r}\|u\|_{\mathcal L^{p,0}(X)}\\
&=Cd(x,y)^{\alpha}\Big(\frac{r}{Cd(x,y)}\Big)^\alpha\log\frac{Cd(x,y)}{r}\|u\|_{\mathcal L^{p,0}(X)}\\
&\le Cd(x,y)^{\alpha}\|u\|_{\mathcal L^{p,0}(X)},
\end{split}
\]
where the last inequality follows by the boundedness of the function 
$f(t)=t^{-\alpha}\log t$ for $t\ge 1$.

If $\beta>0$, then $u$ is $\beta$-H\"older continuous and so
\[
\begin{split}
\big|r^\alpha u_{B_x}-r^\alpha u_{B_y}\big|
&\le r^\alpha\vint{B_x}\vint{B_y}|u(z)-u(w)|\,dz\,dw\\
&\le Cd(x,y)^{\alpha+\beta}\|u\|_{\mathcal L^{p,\beta}(X)}.
\end{split}
\]
{\bf  Case 2.} Suppose then that $r>2d(x,y)$. Let $\Delta_x=B_x\setminus B_y$, 
$\Delta_y=B_y\setminus B_x$ and 
$\Delta=\Delta_x\cup\Delta_y$. As in \cite{M}, we write
\[
\begin{split}
\vint{B_x}u\,d\mu-\vint{B_y}u\,d\mu
&=\frac{1}{\mu(B_x)}\bigg(\int_{B_x}u\,d\mu-\int_{B_y}u\,d\mu
     +\big(\mu(B_y)-\mu(B_x)\big)u_{B_y}\bigg)\\
&=\frac{1}{\mu(B_x)}\bigg(\int_{\Delta_x}u\,d\mu-\int_{\Delta_y}u\,d\mu
     +\big(\mu(\Delta_y)-\mu(\Delta_x)\big)u_{B_y}\bigg)\\
&=\frac{1}{\mu(B_x)}\bigg(\int_{\Delta_x}(u-u_{B_y})\,d\mu
     -\int_{\Delta_y}(u-u_{B_y})\,d\mu\bigg),
\end{split}
\]
which implies that
\begin{equation}\label{est1}
\Big|\,\vint{B_x}u\,d\mu-\vint{B_y}u\,d\mu\Big|
\le C\frac{1}{\mu(B_x)}\int_\Delta|u-u_{B_y}|\,d\mu.
\end{equation}
Suppose that $\beta<0$. 
Let $B_1,B_2,\dots, B_k$ be a maximal collection of disjoint balls of radius $d(x,y)$ centered at $\Delta$. 
Then $\Delta\subset\cup_{i} 2B_i$ and $\cup_{i} B_i\subset \Delta'$, where 
\[
\Delta'=B(x,r+d(x,y))\setminus B(x,r-2d(x,y))\cup B(y,r+d(x,y))\setminus B(y,r-2d(x,y)).
\] 
For each $i$, by \eqref{beta<0} and the H\"older inequality, we have
\begin{equation}\label{est2}
\begin{split}
\vint{2B_i}|u-u_{B_y}|\,d\mu
&\le |u_{2B_i}-u_{B_y}|+\vint{2B_i}|u-u_{2B_i}|\,d\mu\\
&\le Cd(x,y)^\beta \|u\|_{\mathcal L^{p,\beta}(X)}.
\end{split}
\end{equation}
To estimate the measure of $\Delta'$, we first use the annular decay property and the doubling property to obtain
\[
\begin{split}
\mu(B(x,r+d(x,y))\setminus B(x,r-2d(x,y)))
&\le C\Big(\frac{3d(x,y)}{r+d(x,y)}\Big)^\delta\mu(B(x,r+d(x,y)))\\
&\le C\Big(\frac{d(x,y)}{r}\Big)^\delta \mu(B_x),
\end{split}
\]
and similarly, because $\mu(B_{x})$ is comparable with $\mu(B_{y})$ by the doubling property and the assumption $r>2d(x,y)$,
\[
\begin{split}
\mu(B(y,r+d(x,y))\setminus B(y,r-2d(x,y)))
\le C\Big(\frac{d(x,y)}{r}\Big)^\delta \mu(B_x).
\end{split}
\]
Thus
\begin{equation}\label{est3}
\mu(\Delta')\le C\Big(\frac{d(x,y)}{r}\Big)^\delta \mu(B_x).
\end{equation}
Using \eqref{est1}, the facts that $\Delta\subset\cup_{i}2B_{i}$ and $B_{i}\subset\Delta'$ for all $i$, \eqref{est2},  
the disjointedness of the balls $B_{i}$ and \eqref{est3}, we obtain
\[
\begin{split}
|v(x)-v(y)|
&=\Big |\,r^\alpha\vint{B_x}u\,d\mu-r^\alpha\vint{B_y}u\,d\mu\Big |
\le C\frac{r^\alpha}{\mu(B_x)}\int_{\Delta}|u-u_{B_y}|\,d\mu\\
&\le C\frac{r^\alpha}{\mu(B_x)}\sum_{i}\mu(2B_i)\vint{2B_i}|u-u_{B_y}|\,d\mu\\
&\le C\frac{r^\alpha d(x,y)^\beta}{\mu(B_x)}\|u\|_{\mathcal L^{p,\beta}(X)}\sum_{i}\mu(B_i)\\
&\le Cr^\alpha d(x,y)^\beta\frac{\mu(\Delta')}{\mu(B_x)}\|u\|_{\mathcal L^{p,\beta}(X)}\\
&\le Cr^\alpha d(x,y)^\beta\Big(\frac{d(x,y)}{r}\Big)^\delta\|u\|_{\mathcal L^{p,\beta}(X)}\\
&\le Cd(x,y)^{\alpha+\beta}\|u\|_{\mathcal L^{p,\beta}(X)},
\end{split}
\]
where the last inequality follows because $r>d(x,y)$ and $0<\alpha\le\delta$.

Assume then that $\beta=0$ and $0<\alpha<\delta$. 
By \eqref{est1}, the H\"older inequality, the facts that $\Delta\subset 2B_{y}$ and 
$\mu(B_{x})$ is comparable with $\mu(B_{y})$, and \eqref{est3}, we have  
\[
\begin{split}
\Big|\,\vint{B_x}u\,d\mu-\vint{B_y}u\,d\mu\Big|
&\le C\frac{\mu(\Delta)}{\mu(B_x)}^{\alpha/\delta}
\Big(\int_\Delta|u-u_{B_y}|^{\delta/(\delta-\alpha)}\,d\mu\Big)^{1-\alpha/\delta}\\
&\le C\Big(\frac{\mu(\Delta)}{\mu(B_x)}\Big)^{\alpha/\delta}
\Big(\vint{2B_y}|u-u_{B_y}|^{\delta/(\delta-\alpha)}\,d\mu\Big)^{1-\alpha/\delta}\\
&\le C\Big(\frac{d(x,y)}{r}\Big)^{\alpha} \|u\|_{\mathcal L^{\delta/(\delta-\alpha),0}(X)}.
\end{split}
\]
Since, by \eqref{Lp0=BMO}, 
$\|u\|_{\mathcal L^{\delta/(\delta-\alpha),0}(X)}
\le C\|u\|_{\text{BMO}}\le C\|u\|_{\mathcal L^{p,0}(X)}$, 
it follows that
\[
\begin{split}
|v(x)-v(y)|
&=\Big|\,r^\alpha\vint{B_x}u\,d\mu-r^\alpha\vint{B_y}u\,d\mu\Big|
\le Cd(x,y)^{\alpha}\|u\|_{\mathcal L^{p,0}(X)}.
\end{split}
\]
Assume now that $\beta>0$ and $\alpha+\beta\le\delta$. 
Then $u$ is H\"older continuous with exponent $\beta$ and \eqref{Lpbeta=Holder} together with the assumption that $r>d(x,y)$ implies that
\[
|u(z)-u_{B_y}|
\le \vint{B_y}|u(z)-u(w)|\,d\mu(w)
\le Cr^{\beta}\|u\|_{\mathcal L^{p,\beta}(X)}
\]
for each $z\in\Delta$. Hence, using \eqref{est1} and \eqref{est3}, we obtain 
\[
\begin{split}
|v(x)-v(y)|
&\le \frac{r^{\alpha}}{\mu(B_x)}\int_{\Delta}|u-u_{B_y}|d\mu
\le Cr^{\alpha+\beta}\frac{\mu(\Delta)}{\mu(B_x)}\|u\|_{\mathcal L^{p,\beta}(X)}\\
&\le Cr^{\alpha+\beta-\delta}d(x,y)^{\delta}\|u\|_{\mathcal L^{p,\beta}(X)}
\le Cd(x,y)^{\alpha+\beta}\|u\|_{\mathcal L^{p,\beta}(X)}.
\end{split}
\]

Finally, we prove the claim for $\M_{\alpha}u$. 
We may assume that $\M_{\alpha}u(x)\ge\M_{\alpha}u(y)$. 
Let $\eps>0$, and let $r>0$ be such that $r^\alpha u_{B(x,r)}>\M_{\alpha}u(x)-\eps$. 
Then, by the first part of the proof,
\[
\begin{split}
 \M_{\alpha}u(x)-\M_{\alpha}u(y)
 &\le r^\alpha u_{B(x,r)}-r^\alpha u_{B(y,r)}+\eps
 = v(x)-v(y)+\eps\\
 &\le Cd(x,y)^{\alpha+\beta} \|u\|_{\mathcal L^{p,\beta}(X)}+\eps.
\end{split}
\]
The claim follows by letting $\eps\to 0$.
\end{proof}

\begin{remark}
A modification of the proof above shows that the result holds also for the noncentered fractional maximal function
\[
\mathcal{\widetilde M}_\alpha u(x)=\sup_{B(z,r)\ni x}r^\alpha\vint{B(z,r)}|u|\,d\mu.
\]
Let $x,y\in X$. 
We may assume that $u\ge0$ and 
$\mathcal{\widetilde M}_{\alpha}u(x)\ge\mathcal{\widetilde M}_{\alpha}u(y)$.
Let $\eps>0$. Then there exists a ball $B(z,r)$ containing $x$ such that 
\[
\mathcal{\widetilde M}_\alpha u(x)<r_\alpha u_{B(z,r)}+\eps.
\]
Since $y\in B(z,r+d(x,y))$, we have that
\[
\mathcal{\widetilde M}_\alpha u(x)-\mathcal{\widetilde M}_\alpha u(y)
\le r^{\alpha}(u_{B(z,r)}-u_{B(z,r+d(x,y))})+\eps.
\]
Arguments similar to those in the proof of Theorem \ref{main thm} imply that
\[
r^{\alpha}(u_{B(z,r)}-u_{B(z,r+d(x,y))})
\le Cd(x,y)^{\alpha+\beta} \|u\|_{\mathcal L^{p,\beta}(X)}.
\]
The claim follows by letting $\eps\to 0$.
\end{remark}

\section{Sobolev spaces}\label{section: M}
In this section, we show that  the fractional maximal operator $\M_{\alpha}$ maps $L^{p}$-space,$p>1$, to Sobolev spaces, and Sobolev spaces to slightly better Sobolev spaces. 
We prove the results for Sobolev spaces $M^{s,p}(X)$, defined by a pointwise equation. 
These spaces were introduced by Haj\l asz in \cite{Hj1} for $s=1$, and the fractional versions by Yang in \cite{Y}.

Let $s>0$. We say that a measurable function $g\ge0$ is 
{\it a generalized $s$-gradient} of a measurable function $u$, $g\in
\operatorname D^s(u)$, if there is a set $E\subset X$
with $\mu(E)=0$ such that
\begin{equation}\label{m1p}
|u(x)-u(y)|\le d(x,y)^s\big(g(x)+g(y)\big)
\end{equation}
for all $x,y\in X\setminus E$. 
The Sobolev space $M^{s,p}(X)$, $1\le p<\infty$, consists of functions $u\in L^p(X)$ for which there exists a function $g\in L^p(X)\cap \operatorname D^s(u)$. 
The space $M^{s,p}(X)$, equipped with the norm
\begin{equation}\label{m1p norm}
\|u\|_{M^{s,p}(X)}=\bigl(\|u\|_{L^p(X)}^p
                +\inf\|g\|_{L^p(X)}^p\bigr)^{1/p},
\end{equation}
where the infimum is taken over all functions $g\in L^p(X)\cap\operatorname D^s(u)$, is a Banach space \cite[Theorem 8.3]{Hj2}. 

It follows from \eqref{m1p} that every $u\in M^{s,p}(X)$ and $g\in \operatorname D^s(u)$ satisfy the Poincar\'e inequality
\begin{equation}\label{poincare}
\vint{B(x,r)}|u-u_{B(x,r)}|\le Cr^s \vint{B(x,r)}g\,d\mu,
\end{equation}
where $C$ depends only on $s$. Hence the assumption that $X$ supports a Poincar\'e inequality is not needed.  

We will use the following Sobolev type theorem for the fractional maximal operator. 
As in \cite{GPS}, it can be proven easily using the Hardy--Littlewood maximal function theorem; see also \cite{EKM} or \cite{GGKK}.

\begin{theorem}\label{fracM bounded}
Assume that the measure $\mu$ satisfies the lower bound condition \eqref{lower bound}. 
Let $p>1$ and  $0<\alpha<Q/p$. 
There is a constant $C>0$, depending only on 
the doubling constant, constant in the measure lower bound, $p$ and $\alpha$, such that
 \[
 \|\M_{\alpha}u\|_{L^{p^*}(X)}
 \le C\|u\|_{L^p(X)},
 \]
for every $u\in L^{p}(X)$ with $p^*=Qp/(Q-\alpha p)$. 
\end{theorem}

The following theorem is a generalization of the main result of \cite{KS} to the metric setting. 
It shows that the fractional maximal operator is a smoothing operator.
More precisely,  the fractional maximal function of an $L^{p}$-function $u$ has a generalized gradient, and both $\M_{\alpha}u$ and the generalized gradient belong to a higher Lebesgue space than $u$. 

\begin{theorem}\label{Lp for M}
Assume that the measure $\mu$ satisfies the lower bound condition \eqref{lower bound}
and that $X$ satisfies the $\delta$-annular decay property \eqref{B annulus}.
Assume that $u\in L^{p}(X)$ with $1<p<Q$.
Let 
\[
\delta\le\alpha<Q/p,  
\quad p^{*}=Qp/(Q-\alpha p)
\quad\text{and}\quad q=Qp/(Q-(\alpha-\delta) p).
\]
Then there is a constant $C>0$ depending only on the doubling constant and the constant of the $\delta$-annular decay property such that $C\M_{\alpha-\delta}u$ is a generalized $\delta$-gradient of $\M_{\alpha}u$. Moreover, 
 \[
 \|\M_{\alpha}u\|_{L^{p^{*}}(X)}
 \le C\|u\|_{L^p(X)}
 \quad\text{and}\quad
 \|\M_{\alpha-\delta}u\|_{L^q(X)}
 \le C \|u\|_{L^p(X)},
 \]
 where $C$ depends only on the doubling constant, the constant in the measure lower bound,  
$p$ and $\alpha$.
\end{theorem}

\begin{proof}
 We may assume that $u\ge 0$. Let $x,y\in X$. Assume that $\M_{\alpha}u(x)\ge \M_{\alpha}u(y)$. 
 Let $\eps>0$ and let $r>0$ such that
 \[
r^{\alpha}\vint{B(x,r)}u\,d\mu >\M_{\alpha}u(x) - \eps.
 \]
 If $r\le d(x,y)$, then, as $u\ge0$,
 \[
 \M_{\alpha}u(x) -\M_{\alpha}u(y)
\le r^\delta r^{\alpha-\delta}\vint{B(x,r)}u\,d\mu+\eps
\le d(x,y)^\delta\M_{\alpha-\delta}u(x)+\eps.
\]
If $r>d(x,y)$, we write $a=d(x,y)$ and use the doubling property to obtain
\begin{align*}
 \M_{\alpha}u(x) -\M_{\alpha}u(y)
& < r^{\alpha}\vint{B(x,r)}u\,d\mu+\eps-(r+a)^{\alpha}\vint{B(y,r+a)}u\,d\mu\\
&\le r^{\alpha}\bigg(\frac1{\mu(B(x,r))}-\frac1{\mu(B(y,r+a))}\bigg)\int_{B(x,r)}u\,d\mu+\eps\\
&= r^{\alpha}\frac{\mu(B(y,r+a))-\mu(B(x,r))}{\mu(B(y,r+a))}\vint{B(x,r)}u\,d\mu+\eps\\
&\le Cr^{\alpha}\frac{\mu(B(x,r+2a)\setminus B(x,r))}{\mu(B(x,r+2a))}\vint{B(x,r)}u\,d\mu +\eps.
\end{align*}
The $\delta$-annular decay property together with the assumption $r>d(x,y)$ implies that
\[
\frac{\mu(B(x,r+2a)\setminus B(x,r))}{\mu(B(x,r+2a))}
\le K\left(\frac a{r+2a}\right)^\delta 
\le K\left(\frac{d(x,y)}r\right)^\delta,
\]
and hence
\[
\M_{\alpha}u(x) -\M_{\alpha}u(y)
 \le C d(x,y)^\delta\M_{\alpha-\delta}u(x)+\eps.
\]
By letting $\eps\to0$ and changing the roles of $x$ and $y$, we have that 
\[
| \M_{\alpha}u(x) -\M_{\alpha}u(y)|
 \le C d(x,y)^{\delta}\big(\M_{\alpha-\delta}u(x)+\M_{\alpha-\delta}u(y)\big).
\]
Hence $C\M_{\alpha-\delta}u$ is a generalized $\delta$-gradient of $\M_{\alpha}u$.
The norm estimates for $\M_{\alpha}u$ and  $\M_{\alpha-\delta}u$  follow from Theorem 
\ref{fracM bounded}.
\end{proof}

\begin{remark}
 Under the assumptions of Theorem \ref{Lp for M}, it holds true that 
 $\M_{\alpha}u\in M^{\delta,q}_{\text{loc}}(X)$ and 
 \[
 \|\M_{\alpha}u\|_{M^{\delta,q}(A)}
 \le\mu(A)^{1/q-1/p^{*}}\|u\|_{L^p(A)}
 \]
 for all open sets $A\subset X$ with $\mu(A)<\infty$.
\end{remark}

\begin{remark}
The above proof with the ball $B(x,r)$ replaced by a ball $B(z,r)$ for which 
$\mathcal{\widetilde M}u(x)<r^\alpha u_{B(z,r)}+\eps$  
and $B(y,r+a)$ replaced by the ball $B(z,r+a)$ shows that the result holds also for
the noncentered fractional maximal function.
\end{remark}

In the next theorem, we show that if $u$ is a Sobolev function, then its fractional maximal function belongs to a Sobolev space with the Sobolev conjugate exponent. 
The proof is a modification of the result that the usual Hardy--Littlewood maximal operator is bounded in Sobolev spaces if the underlying space satisfies the relative $1$-annular decay property.
Since the original proof of MacManus in \cite{M2} (for $\alpha=0$) is unpublished, we give in the proof below all the details.

\begin{theorem}\label{thm: sobo}
Assume that the measure $\mu$ satisfies the lower bound condition and $X$ satisfies the relative 
$1$-annular decay property \eqref{MM annulus}. 
Let $p>1$, $u \in M^{1,p}(X)$ and $0<\alpha<Q/p$. 
Then $\M_{\alpha} u \in M^{1,p^{*}}(X)$ with $p^{*}= Qp/(Q - \alpha p)$
and there is a constant $C>0$, depending only on the doubling constant, the constant in the measure lower bound, $p$ and $\alpha$, such that
\begin{equation}\label{rajoittuneisuus}
\|\M_{\alpha} u\|_{M^{1,p^{*}}(X)} 
\le C\|u\|_{M^{1,p}(X)}.
\end{equation}
\end{theorem}

\begin{proof}
Let $u \in M^{1,p}(X)$ and let $g\in L^p(X)$ be a generalized gradient of $u$. 
We may assume that $u \geq 0$ since $|u| \in M^{1,p}(X)$ and $g$ is a generalized gradient of $|u|$. 
Fix $1<q<p$ and define 
\[
\tilde{g}= \big(\M_{\alpha q} (g^{q})\big)^{1/q}. 
\]
Since $p/q>1$, Theorem \ref{fracM bounded} implies that
\[
\|\tilde{g}\|_{p^{*}} 
= \|\M_{\alpha q} (g^{q})\|^{1/q}_{\frac{Q\frac{p}{q}}{Q-(\alpha q) \frac{p}{q}}} 
\le C \|g^{q}\|^{1/q}_{\frac{p}{q}} = C\|g\|_{p}.
\]
If we can show that $\tilde g$ is a generalized gradient of $\M_{\alpha} u$, then this together with 
Theorem \ref{fracM bounded} for $u$ implies norm estimate \eqref{rajoittuneisuus}. 

We are going to show that
\begin{equation}\label{todistettava}
\Big |r^{\alpha}\vint{B (x,r)} u \, d\mu - r^{\alpha}\vint{B (y,r)} u \, d\mu\Big| 
\leq  C d(x,y) (\tilde{g}(x) +\tilde{g}(y))
\end{equation}
for almost all $x$, $y \in X$ and all $r>0$. This implies that
\begin{equation}\label{gradientti}
|\M_{\alpha} u(x) - \M_{\alpha} u(y)| 
\leq C d(x,y)(\tilde{g}(x) +\tilde{g}(y))
\end{equation}
for almost every $x$, $y \in X$, which proves our theorem.
 
The proof of \eqref{todistettava} is easy if $r \le 3d(x,y)$:
Since $g$ is a generalized gradient of $u$,
\[
|u(z)-u(w)| 
\le Cd(z,w)(g(z) + g(w)) 
\le C d(x,y)(g(z) + g(w))
\]
for almost all $z \in B(x,r)$ and $w \in B(y,r)$. 
By integrating both sides with respect to $z$ and $w$ and using the H\"older inequality, we obtain
\begin{align*}
\Big | \, \vint{B(x,r)} u \, d\mu - \vint{B (y,r)} u \, d\mu\Big| 
\leq  C \,d(x,y) \bigg( \Big( \, \vint{B (x,r)} g^{q} \, d\mu \Big)^{1/q}
+ \Big( \, \vint{B(y,r)} g^{q} \, d\mu \Big)^{1/q} \bigg).
\end{align*}
Now (\ref{todistettava}) follows by multiplying both sides by $r^{\alpha}$ and using the definition of 
$\tilde g$.

Suppose then that $r> 3d(x,y)$. 
Let $\Delta_{y} = B(y,r) \setminus B(x,r)$, $\Delta_{x} = B(x,r) \setminus B(y,r)$ and 
$\Delta = \Delta_{y} \cup \Delta_{x}$. 
As in the proof of Theorem \ref{main thm}, we have 
\[
\vint{B(y,r)} u \, d\mu - \vint{B(x,r)} u \, d\mu 
= \frac1{\mu(B(y,r))}\int_{\Delta_{y}}(u-u_{B(x,r)}) \, d\mu - \int_{\Delta_{x}}(u-u_{B(x,r)}) \, d\mu,
\]
and hence
\[
\Big |\, \vint{B(x,r)} u \, d\mu - \vint{B(y,r)} u \, d\mu \Big | 
\leq \frac{1}{\mu(B(y,r))} \int_{\Delta} |u-u_{B(x,r)} | \, d\mu.
\]
Now, let 
\[
A=B(x,r+d(x,y))\setminus B(x,r-d(x,y)).
\] 
Since the balls $B(x,r)$ and $B(y,r)$ have comparable measures and $\Delta \subset A$, 
\[
\Big| \, \vint{B(x,r)} u \, d\mu - \vint{B(y,r)} u \, d\mu \Big| 
\leq \frac{C}{\mu(B(x,r))} \int_{A} |u-u_{B(x,r)} | \, d\mu.
\]
We want to show that the right side is bounded by
$C d(x,y) \, (\vint{B(x,5r)} g^{q} \, d\mu)^{1/q}$, which implies (\ref{todistettava}). 
For that, we prove the following estimate for the integral over the annulus.
\medskip

\noindent {\bf Claim:}
\[
\frac{1}{\mu(B(x,r))} \int_{A} |u-u_{B(x,r)} | \, d\mu 
\leq C \, d(x,y) \, \vint{B(x,5r)} g(w) \log \frac{5 \, r}{|r - d(w,x)|} \, d\mu.
\]
To prove the claim, define $r_{k}=3^{k} r$ and $B_{k}(z)=B(z,r_{k})$, $k\in\z$. 

We use a standard chaining argument, Poincar\'e inequality \eqref{poincare}, and the fact that 
$B(x,r) \subset B_{1}(z)$ for each $z \in A$ to see that 
\[
|u(z) - u_{B(x,r)}|
\le|u(z) - u_{B_1(z)}|+|u_{B_1(z)} - u_{B(x,r)}|
\leq C \, \sum_{k \leq 1} r_{k} \vint{B_{k}(z)} g \, d\mu
\]
for all Lebesgue points $z \in A$. (Since almost every point is a Lebesgue point of $u$, this holds for almost all $z\in A$.)

Integration of both sides over $A$ and a use of the Fubini theorem yield
\[
\int_{A} |u(z) - u_{B(x,r)}| \, d\mu (z) 
\leq C \int_{X} g(w) \, K(w) \, d\mu (w),
\]
where
\[
K(w)= 
\int_{A} \Big(\sum_{k \leq 1} \frac{r_{k}}{\mu(B_{k}(z))} \chi_{B_{k}(z)} (w) \Big)\,d\mu (z).
\]
Since $k \leq 1$, we have that $K(w)=0$ when $w \notin B(x,5r)$. This implies that
\[
\int_{A} |u(z) - u_{B(x,r)}| \, d\mu (z) 
\leq C \int_{B(x,5r)} g(w) \, K(w) \, d\mu (w),
\]
where, by the fact that $\chi_{B_{k}(z)} (w) = \chi_{B_{k}(w)} (z)$,
\[
K(w)=
\int_{A} \Big(\sum_{k \leq 1} \frac{r_{k}}{\mu(B_{k}(z))} \chi_{B_{k}(w)} (z) \Big)\,d\mu (z).
\]
If $z \in B_{k}(w)$, then the balls $B_{k}(z)$ and $B_{k}(w)$ have comparable measures. Thus
\[
\int_{A} \frac{r_{k}}{\mu(B_{k}(z))} \chi_{B_{k}(w)} (z) \, d\mu (z) 
\leq C r_{k} \frac{\mu(A \cap B_{k}(w))}{\mu (B_{k}(w))},
\]
from which we obtain that
\[
K(w) 
\leq C \sum_{k \leq 1} r_{k} \frac{\mu(A \cap B_{k}(w))}{\mu(B_{k}(w))}.
\]
It follows from the relative $1$-annular decay that each term in the above sum is bounded by 
$C\min\{r_{k}, d(x,y)\}$. 
Moreover, for the indices $k \leq 1$ for which $B_{k}(w) \cap A= \emptyset$, the terms are zero. Now
\[
K(w) 
\leq C \sum_{k\in\mathcal K_1}r_k + C \sum_{k\in\mathcal K_2} d(x,y),
\]
where $\mathcal K_1=\{k\le1:r_{k} \le d(x,y)\}$ and 
$\mathcal K_2=\{k\le1:r_{k}>d(x,y),\,B_{k}(w) \cap A\ne \emptyset\}$.
The first term is at most $C d(x,y)$. 
An upper bound for the second term is
\begin{equation}\label{number of terms}
Cd(x,y) \log\frac{5r}{|r - d(w,x)|}. 
\end{equation}
To see this, we consider two cases. 
Assume first that $|r - d(w,x)| < 3d(x,y)$. Since $r_{k} > d(x,y)$, we have that 
\[
k > \frac{1}{\log 3} \log\frac{d(x,y)}{r}
> \frac{1}{\log 3} \log\frac{|r-d(w,x)|}{3r},
\] 
which implies that the amount of such indices in the sum is less than 
\[
C\log \frac{5r}{|r - d(w,x)|}.
\] 

The second case, $|r-d(w,x)| \geq 3 \, d(x,y)$, can be split to two parts. 
When $w \in B(x, r-3d(x,y))$, the requirement that $B_k(w)\cap A\ne\emptyset$ implies that
\[
3^kr=r_{k}\ge r - d(x,y) - d(w,x),
\]
and hence we have at most
\[
C\log\frac{r}{r-d(x,y)-d(w,x)}
\le C\log\frac{2r}{|r-d(w,x)|}
\]
such terms. (Note that the condition $B_k(w)\cap A\ne\emptyset$ gives the number of terms, not the condition $r_k>h$.)
Similarly, when $w \in 5B(x,r) \setminus B(x, r+3d(x,y))$, we have
\[
3^kr=r_{k}\ge d(w,x) - (r+d(x,y)),
\]
and the upper bound for the number of terms follows just as in the first case.
This implies that the second term is bounded above by \eqref{number of terms}.
Since $0 \le d(w,x) < 5r$, we have $\log \frac{5r}{|r - d(w,x)|}\ge\log\frac54$, and hence the claim follows.

Now, by the claim and the H\"{o}lder inequality,
\[
\frac{1}{\mu(B(x,r))} \int_{A} |u-u_{B(x,r)} | \, d\mu  
\leq Cd(x,y) \Big( \, \vint{B(x,5r)} g^{q} \, d\mu\Big)^{1/q}
\Big( \, \vint{B(x,5r)} L^{q'} \, d\mu\Big)^{1/q'},
\]
where $q'$ is the conjugate exponent of $q$ and 
\[
L(w)=\log \frac{5r}{|r - d(w,x)|}.
\]
To estimate the integral of $L^{q'}$ over $B(x,5r)$, we define for each $i=0,1,\dots$ 
\[
A_{i}= \{w \in B(x, 5r): 4^{-i}r \leq |r - d(w,x)| < 4^{-i+1}r\}.
\]
The sets $A_{i}$ are disjoint, and on each $A_{i}$ we have
\[
L(w) \le C(1+i).
\]
Moreover, the set $\{ w\in B(x, 5r):|r-d(w,x)|=0\}$ has measure zero by the relative annular decay, 
and so
\[
\mu\bigg(B(x,5r)\setminus\bigcup_{i=0}^{\infty} A_{i}\bigg)  = 0.
\]
It follows that
\[
\vint{B(x,5r)} L^{q'} \, d\mu 
\le C\sum_{i=0}^\infty \frac{(1+i)^{q'} \mu(A_{i})}{\mu(B(x,r))}.
\]
For $i \geq 2$, the set $A_{i}$ consists of two annuli of thickness $4^{-i}r$ centered at $x$. 
The inner and outer radii of these annuli are comparable to $r$. Thus the relative $1$-annular decay implies that
\[
\mu(A_{i}) \leq C \, 4^{-i} \mu(B(x,r)).
\]
The same estimate is trivial when $i=0$ or $i=1$. It follows that
\[
\vint{B(x,5r)} L^{q'} \, d\mu 
\le C\sum_{i=0}^\infty \frac{(1+i)^{q'}}{4^{i}}.
\]
This sum converges, which can be seen for example by a ratio test, and we have that
\begin{align*}
 \frac{1}{\mu(B(x,r))} \int_{A} |u-u_{B(x,r)} | \, d\mu 
&\le C \, d(x,y) \Big( \, \vint{B(x,5r)} g^{q} \, d\mu\Big)^{1/q}.
\end{align*}
Estimate (\ref{todistettava}) and hence the theorem follows from this.
\end{proof}

We close this section by considering the more general case where $u$ belongs to a fractional Sobolev space and $X$ satisfies the relative $\delta$-annular decay property. Using similar arguments as in the proof above, we obtain the following results.

\begin{theorem}\label{thm: sobo2}
Assume that $X$ satisfies the relative $\delta$-annular decay property \eqref{MM annulus}. 
Let $\alpha>0$,$1<q<p$, $s>0$, $u \in M^{s,p}(X)$ and $g\in\operatorname D^s(u)$. Then there is a constant $C>0$ such that
the following holds.
\begin{itemize}
\item[a)] If $s<\delta$, then
\[
\tilde{g}= C\M_{\alpha}g
\]
is a generalized $s$-gradient of $\M_\alpha u$.

\item[b)] If $s=\delta$, then 
\[
\tilde{g}= C(\M_{\alpha q} (g^{q}))^{1/q} 
\]
is a generalized $s$-gradient of $\M_\alpha u$.
\item[c)] If $s>\delta$, then
\[
\tilde{g}= C\M_{\alpha+s-\delta}g
\]
is a generalized $\delta$-gradient of $\M_\alpha u$.
\end{itemize}
\end{theorem}
\begin{proof} It suffices to prove the claim for the functions $x\mapsto r^\alpha u_{B(x,r)}$, $r>0$. 
Fix $x,y\in X$ and $r>0$. Suppose first that $r\le 3d(x,y)$. 
If $s\le\delta$, the desired estimate follows as in the proof of Theorem \ref{thm: sobo}. 
If $s>\delta$,  we need a simple chaining argument. 
Let $k$ be the smallest integer such that $2^kr\ge 4d(x,y)$ and let $z\in\{x,y\}$. 
Then the doubling property, the fact that 
$B(x,3d(x,y))\subset B(z,2^kr)$ and the Poincar\'e inequality \eqref{poincare} imply that
\[
\begin{split}
|u_{B(z,r)}-u_{B(x,3d(x,y))}|
&\le \sum_{i=0}^{k-1}|u_{B(z,2^ir)}-u_{B(z,2^{i+1}r)}|+|u_{B(z,2^kr)}-u_{B(x,3d(x,y))}|\\
&\le C\sum_{i=1}^k\vint{B(z,2^ir)}|u-u_{B(z,2^ir)}|\,d\mu\\
&\le C\sum_{i=1}^k(2^ir)^s\vint{B(z,2^ir)}g\,d\mu.
\end{split}
\]
Hence, using the selection of $k$, we have that
\[
\begin{split}
r^\alpha|u_{B(z,r)}-u_{B(x,3d(x,y))}|
&\le C\sum_{i=1}^k(2^ir)^{\alpha+s}\vint{B(z,2^ir)}g\,d\mu\\
&\le C\sum_{i=1}^k(2^ir)^{\delta}\M_{\alpha+s-\delta}g(z)\\
&\le Cd(x,y)^\delta \M_{\alpha+s-\delta}g(z),
\end{split}
\]
and so
\[
\begin{split}
|r^\alpha u_{B(x,r)}-r^\alpha u_{B(y,r)}|
&\le r^\alpha|u_{B(x,r)}-u_{B(x,3d(x,y))}|+r^\alpha|u_{B(y,r)}-u_{B(x,3d(x,y))}|\\
&\le Cd(x,y)^\delta\left(\M_{\alpha+s-\delta}g(x)+\M_{\alpha+s-\delta}g(y)\right).
\end{split}
\]

Suppose then that $r>3d(x,y)$.
By following the proof of Theorem \ref{thm: sobo}, we obtain the estimate
\begin{equation}\label{estimate}
|u_{B(x,r)}-u_{B(y,r)}|\le C\vint{B(x,5r)}g(w)K(w)\,d\mu(w),
\end{equation}
where
\[
K(w) 
\leq C \sum_{k \leq 1} r_{k}^s\ \frac{\mu(A \cap B_{k}(w))}{\mu(B_{k}(w))}.
\]
The relative $\delta$-annular decay  
implies that
\[
K(w) 
\leq C\Big( \sum_{k\in\mathcal K_1}r_k^s 
+  d(x,y)^\delta\sum_{k\in\mathcal K_2}r_k^{s-\delta}\Big) ,
\]
where $\mathcal K_1=\{k\le1:r_{k} \le d(x,y)\}$ and 
$\mathcal K_2=\{k\le1:r_{k}>d(x,y),\,B_{k}(w) \cap A\ne \emptyset\}$.
It follows that
\[
K(w)\le C
\begin{cases}
d(x,y)^s, &\text{ if } s<\delta\\
d(x,y)^s \log\frac{5r}{|r - d(w,x)|}, &\text{ if } s=\delta\\
d(x,y)^{\delta}r^{s-\delta}, &\text{ if } s>\delta\\
\end{cases}.
\]
If $s<\delta$ or $s>\delta$, the claim follows by combining the above estimate with \eqref{estimate} and multiplying the resulting inequality by $r^\alpha$.
In the case $s=\delta$, we argue as in the proof of Theorem \ref{thm: sobo}.
\end{proof}

\begin{corollary}
Assume that the measure $\mu$ satisfies the lower bound condition and $X$ satisfies the relative 
$\delta$-annular decay property \eqref{MM annulus}. 
Let $p>1$, $s>0$ and $u \in M^{s,p}(X)$. 

If $s\le\delta$ and $0<\alpha<Q/p$, then $\M_{\alpha} u \in M^{s,p^{*}}(X)$ with $p^{*}= Qp/(Q - \alpha p)$
and there is a constant $C>0$, independent of $u$, such that
\[
\|\M_{\alpha} u\|_{M^{s,p^{*}}(X)} \le C \, \|u\|_{M^{s,p}(X)}.
\]
If $s\ge\delta$ and $\alpha+s-\delta<Q/p$, then $\M_{\alpha} u \in\dot M^{\delta,q}(X)\cap L^{p^*}(X),$ where $q= Qp/(Q - (\alpha+s-\delta)p)$ and $\dot M^{\delta,q}(X)$ is the homogeneous Hajlasz space equipped with the seminorm $\|u\|_{\dot M^{\delta,q}(X)}=\inf_{g\in \operatorname D^\delta(u)}\|g\|_{L^{q}(X)}$.
Moreover, there is a constant $C$, independent of $u$, such that
\[
\|\M_{\alpha} u\|_{\dot M^{\delta,q}(X)}+ \|\M_{\alpha} u\|_{L^{p*}(X)} \le C \, \|u\|_{M^{s,p}(X)}.
\]
 \end{corollary}

\begin{proof}
The claim follows from Theorems \ref{thm: sobo2} and \ref{fracM bounded}.
\end{proof}

For related results concerning the discrete fractional maximal operator, see \cite{HT}.

\begin{remark}
The usual modifications in the proofs show that Theorems \ref{thm: sobo} and \ref{thm: sobo2}
remain true for the noncentered fractional maximal function.
\end{remark}

\section{Examples}\label{section: example}
We modify the example given by Buckley \cite{B} a little bit and show that the fractional maximal function of  a Lipschitz function may fail to be continuous.

\begin{example}
First we recall Buckley's example. 
Let $X$ be the subset of the complex plane consisting of the real line and the points $x$ on the unit circle whose argument $\theta$ lies in the interval $[0, \frac{\pi}{2}]$. 
Equip $X$ with the Euclidean metric and the $1$-dimensional Hausdorff measure. 
Let $u\colon X\to[0,1]$ be a Lipschitz function such that $u(x)=0$, if 
$x\in\re$ or $\text{Arg}(x)\le \pi/5$, and
$u(x)=1$, if $\text{Arg}(x)\ge \pi/4$. Then $\M u$ has a jump discontinuity at the origin.
Indeed, since 
\[
\M u(0)=\lim_{r\to 1+}\vint{B(0,r)}u\,d\mu
=\frac{1}{2+\frac{\pi}{2}}\int_{\overline{B}(0,1)}|u|\,d\mu,
\]
we have that
\[
\M u(0)\le\frac{\frac{\pi}{2}-\frac{\pi}{5}}{2+\frac{\pi}{2}}=\frac{3\pi}{20+5\pi}.
\]
If $x<0$,  then $B(x,r(x))$, where  $r(x)=d(x,e^{i\pi/4})$,  includes points on the arc if and only if their argument exceeds 
$\frac{\pi}{4}$. It follows that
\[
\lim_{x\to 0-}\M u(x)\ge \lim_{x\to 0-}\vint{B(x,r(x))}|u|\,d\mu
=\frac{\frac{\pi}{4}}{2+\frac{\pi}{4}}=\frac{\pi}{8+\pi}>\M u(0).
\]

The first part of the above argument does not work for $\M_\alpha$, because for $\alpha$ large enough, 
\[
r^\alpha\vint{B(0,r)}|u|\,d\mu
=\frac{r^\alpha}{2r+\frac{\pi}{2}}\int_{\overline{B}(0,1)}|u|\,d\mu
\]
no longer maximizes as $r\to 1+$.
This difficulty can be overcome by modifying the measure. 
Let $w\colon X\to\re$ be a weight such that $w(x)=1+\frac{\pi}{2}$, when $x>1$ and $w(x)=1$ otherwise. 
Then the measure $\nu$, defined by
\[
\nu(A)=\int_Aw\,d\mu, 
\]
is doubling and satisfies $\nu(B(0,r))=(2+\frac{\pi}{2})r$, for $r>1$.
If $\alpha>1$, it follows that $\M_\alpha u\equiv\infty$.
If $\alpha\le 1$, then, for $r>1$,
\[
r^\alpha\vint{B(0,r)}u\,d\nu
=\frac{r^\alpha}{(2+\frac{\pi}{2})r}\int_{\overline{B}(0,1)}|u|\,d\mu
\le \frac{1}{2+\frac{\pi}{2}}\int_{\overline{B}(0,1)}|u|\,d\mu,
\]
which implies that 
\[
\M_\alpha u(0)
=\frac{1}{2+\frac{\pi}{2}}\int_{\overline{B}(0,1)}|u|\,d\mu.
\]
The rest of the argument is the same as above.
\end{example}

The next example shows that also the noncentered fractional maximal function of a Lipschitz function
is not necessarily continuous.
\begin{example}
Let $X=(\mathbb R\times\{0\})\cup (\{0\}\times (-\infty,1])\subset\re^2$. 
Equip $X$ with the metric $d(x,y)=\max\{|x_1-y_1|,|x_2-y_2|\}$ and the $1$-dimensional Hausdorff measure $\mu$. 
Define $u\colon X\to\mathbb R$ by setting $u(x)=x_2$ for $0< x_2\le 1$ and $u(x)=0$ otherwise.
Then it is easy to see that $\mathcal{\widetilde M}_\alpha u$, $0\le\alpha\le 1$, is not continuous at the origin.
If $x_2>0$, then clearly $\mathcal{\widetilde M}_\alpha u(x)\ge\frac12$.
We will show that $\mathcal{\widetilde M}_\alpha u(0)\le\frac13$.
Suppose that $B(x,r)$ contains the origin. If $r\le\frac12$, then 
\[
r^\alpha\vint{B(x,r)}u\,d\mu\le \frac{r^\alpha}{4r}\cdot 2r^2\le\frac14.
\]
If $r>\frac12$, then 
\[
r^\alpha\vint{B(x,r)}u\,d\mu\le \frac{r^\alpha}{3r}\cdot\frac12\le\frac13.
\]
Hence $\mathcal{\widetilde M}_\alpha u$ is not continuous at the origin.
\end{example}

\subsection*{Acknowledgements}The research is supported  by
the Academy of Finland, grants no.\ 135561 and 252108.

\vspace{0.5cm}
\noindent
\small{\textsc{T.H.},}
\small{\textsc{Department of Mathematics},}
\small{\textsc{P.O. Box 11100},}
\small{\textsc{FI-00076 Aalto University},}
\small{\textsc{Finland}}\\
\footnotesize{\texttt{toni.heikkinen@aalto.fi}}

\vspace{0.3cm}
\noindent
\small{\textsc{J.L.},}
\small{\textsc{Department of Mathematics and Statistics},}
\small{\textsc{P.O. Box 35},}
\small{\textsc{FI-40014 University of Jyv\"askyl\"a},}
\small{\textsc{Finland}}\\
\footnotesize{\texttt{juha.lehrback@jyu.fi}}

\vspace{0.3cm}
\noindent
\small{\textsc{J.N.},}
\small{\textsc{Department of Mathematics and Statistics},}
\small{\textsc{P.O. Box 35},}
\small{\textsc{FI-40014 University of Jyv\"askyl\"a},}
\small{\textsc{Finland}}\\
\footnotesize{\texttt{juho.nuutinen@jyu.fi}}

\vspace{0.3cm}
\noindent
\small{\textsc{H.T.},}
\small{\textsc{Department of Mathematics and Statistics},}
\small{\textsc{P.O. Box 35},}
\small{\textsc{FI-40014 University of Jyv\"askyl\"a},}
\small{\textsc{Finland}}\\
\footnotesize{\texttt{heli.m.tuominen@jyu.fi}}

\end{document}